\documentclass[12pt]{amsart}

\usepackage{palatino}
\usepackage{amsmath,amssymb,amsthm,amscd, amsxtra, amsfonts,graphicx,fancybox,comment}
\usepackage[all]{xy}
\usepackage{amssymb}
\usepackage{latexsym}
\usepackage{amscd}
\usepackage{color}
\input amssym.def
\input amssym

\newtheorem{theorem}{Theorem}[section]

\newtheorem{lemma}[theorem]{Lemma}
\newtheorem{proposition}[theorem]{Proposition}
\newtheorem{remark}[theorem]{Remark}
\newtheorem{definition}[theorem]{Definition}

\newtheorem{conjecture}[theorem]{Conjecture}

\newcommand{\triplet}{\mathcal{W}(p)}
\newcommand{\singlet}{\overline{M(1)} }
\newcommand{\bea}{\begin{eqnarray}}

\newcommand{\eea}{\end{eqnarray}}

\newcommand{\N}{\Bbb N}
\newcommand{\Z}{\mathbb{Z}}

\begin{document}

\title[Applications and constructions of intertwining operators]{Some applications and constructions of intertwining operators in LCFT}

\author{Dra\v{z}en Adamovi\'c and Antun Milas}
\address{Department of Mathematics, University of Zagreb, Croatia}
\email{adamovic@math.hr}

\address{Department of Mathematics and Statistics,
University at Albany (SUNY), Albany, NY 12222}
\email{amilas@math.albany.edu}
\thanks{A.M. was partially supported by a Simons Foundation grant ($\#$ 317908) \\ D.A. is  partially supported by the Croatian Science Foundation under the project 2634 }

\begin{abstract}
We discuss some applications of fusion rules and intertwining operators in the representation theory of cyclic orbifolds of the triplet vertex operator algebra. We prove that the classification of irreducible modules for the orbifold vertex algebra $\triplet^{A_m}$ follows from a conjectural fusion rules formula for the singlet vertex algebra modules. In the $p=2$ case, we computed fusion rules for the irreducible singlet vertex algebra modules by using intertwining operators. This result implies the classification of irreducible modules for $\mathcal{W}(2)^{A_m}$, conjectured previously in \cite{ALM-CCM}. The main technical tool is a new deformed realization of the triplet and singlet vertex algebras, which is used to construct certain intertwining operators that can not be detected by using standard free field realizations. 
\end{abstract}

\maketitle

\section{Introduction}

The orbifold theory with respect to a finite group of automorphisms is an important part of vertex algebra theory. 
 There is a vast literature on this subject in the case of familiar rational vertex algebras, including lattice vertex algebras.
It is widely believed that all irreducible representations of the orbifold algebra $V^G$ ($G$ is finite) arise from the ordinary and $g$-twisted $V$-modules via restriction to the subalgebra.
This is known to hold if $V^G$ is regular  \cite{Dong}. In fact, one reason for a great success of orbifold theory in the regular case 
comes from the underlying modular tensor category structure (e.g., quantum dimensions, Verlinde formula, etc.).

For irrational $C_2$-cofinite vertex algebra this theory is much less developed and only a few known examples have been studied, such as the $\mathbb{Z}_2$-orbifold
of the symplectic fermion vertex operator superalgebra \cite{Abe} and ADE-type orbifolds for the triplet vertex algebra $\triplet$. As we had demonstrated in \cite{ALM-CCM}-\cite{ALM-SIGMA} (jointly with X.Lin), for the triplet vertex algebra, we have a very strong evidence that all irreducible modules for $\triplet^G$ arise in this way.  Let us briefly summarize the main results obtained in those papers.
For the A-type (cyclic) orbifolds, based on the analysis of the Zhu algebra, we obtain classification of irreducible modules 
for $\triplet^{A_2}$ for small $p$ \cite{ALM-CCM}. After that, in \cite{ALM-IJM}, we studied the $D$-series (dihedral) orbifolds of $\triplet$. Among many results, we classified irreducible 
$\triplet^{D_2}$-modules again for small $p$. Finally, in \cite{ALM-SIGMA} we reduced the classification for all $A$ and $D$-type orbifolds
to a combinatorial problem related to certain constant term identities. However, these constant term identities are very difficult to prove even for 
moderately small values of $p$.

In this paper we revisited the classification of modules for the cyclic orbifolds by using a different circle of ideas. Instead of working with the Zhu 
algebra we employ the singlet vertex subalgebra to classify irreducible $\triplet^{A_m}$-modules. Here is the content of the paper with the main results.

In Section 2, we review several basic notions pertaining to singlet, doublet and triplet vertex algebras, including $A_n$-orbifolds. After that, we prove that the 
classification problem for $\triplet^{A_n}$, as conjectured in \cite{ALM-CCM}, follows from a certain fusion rules conjecture for 
the singlet vertex algebra module and simple currents.  In Section 3, we study fusion rules for the singlet vertex algebra for $p=2$. By using Zhu's algebra we obtain an upper bound on the fusion rules of typical modules. In Section 4, which is the most original part of the paper, we obtained a new realization of the doublet (and thus of the singlet and triplet) vertex algebra for $p=2$ and $p=3$. We also have a conjectural realization for all $p \geq 4$.  Finally, in Section 5, by using the construction 
in Section 4, we construct all intertwining operators among singlet modules when $p=2$. Combined with the previously obtained upper bound for generic modules, this proves that 
some special singlet algebra modules are simple currents, in a suitable sense. Using these results, we prove the completeness of classification of irreducible modules for the orbifold 
algebra $\mathcal{W}(2)^{A_m}$.

\smallskip


\section{ Classification of irreducible modules for orbifold $\triplet ^{A_m}$:  A fusion rules approach}

Let $L=\mathbb{Z}\alpha$ be a~rank one lattice with $\langle\alpha,\alpha\rangle=2p$ ($p\geq1$).
Let
\begin{gather*}
V_{L}=\mathcal{U}(\widehat{\mathfrak{h}}_{<0})\otimes\mathbb{C}[{L}]
\end{gather*}
denote the corresponding lattice vertex  algebra~\cite{LL}, where $\mathfrak{\widehat{h}}$ is the af\/f\/inization of
$\mathfrak{h}=\mathbb{C}\alpha$, and $\mathbb{C}[{L}]$ is the group algebra of~$L$.
Let $M(1)$ be the Heisenberg vertex subalgebra of $V_L$ generated by $\alpha(-1) \textbf{1}$
with the conformal vector
\begin{gather*}
\omega=\frac{\alpha(-1)^{2}}{4p}\textbf{1}+\frac{p-1}{2p}\alpha(-2)\textbf{1}.
\end{gather*}
With this choice of $\omega$, $V_{L}$ has a~vertex operator algebra structure of central charge
\begin{gather*}
c_{1,p}=1-\frac{6(p-1)^{2}}{p}.
\end{gather*}
Let  $Q:={\rm Res}_x Y(e^\alpha,x)=e^{\alpha}_0$,  where we
use $e^\beta$ to denote vectors in the group algebra of the dual lattice of~$L$.
The triplet vertex algebra $\mathcal{W}(p)$ 
 is strongly generated by~$\omega$ and three primary vectors
\begin{gather*}
F=e^{-\alpha}, \qquad  H=Q F, \qquad  E=Q^2 F
\end{gather*}
of conformal weight $2 p-1$ \cite{AdM-2008}. We use $H(n):={\rm Res}_x x^{2p-2+n} Y(H,x)$

The doublet vertex algebra $\mathcal A(p)$ (see \cite{FFT, AdM-2013}) is a generalized vertex algebra strongly generated by $\omega$ and
$$ {\overline a} ^- = e^{-\alpha /2}, \ \   {\overline a}^  +=  Q e^{-\alpha /2}. $$
We use $\overline{M(1)}$ to denote the singlet vertex algebra.
It is defined as a vertex subalgebra of $\mathcal{W}(p)$ generated by~$\omega$ and~$H$.

It was proved in \cite{A-2003} that the irreducible $\N$--graded $\singlet$--modules are parametrized  by highest weights with respect to $(L(0), H(0))$ which has the form
$$ \lambda_t :=\left(\frac{1}{4p} t (t+ 2p-2), {t \choose 2p-1}\right). $$
The irreducible  $\singlet$-module $M_t$  with highest weight $\lambda_t$ is realized as an irreducible subquotient of the $M(1)$--module $$M(1, \frac{t}{2p} \alpha):= M(1) \otimes e^{\tfrac{t}{2p} \alpha} . $$

For $n \in {\N}$, we define:
  $$ \pi_n :=M_{-np}=  \overline{M(1)}. e ^{- n \alpha /2} \quad \mbox{ and } \quad \pi_{-n} := M_{np}= \overline{M(1)}.  Q ^{n} e ^{-n \alpha /2}. $$
  We know \cite{AdM-2008} that $\pi_j $, $j \in {\Z}$, are irreducible as $\overline{M(1)}$--modules and $\pi_j \subset M(1, - j \alpha /2 )= M(1) \otimes e ^{- j \alpha /2}. $
  
Recall that the vertex algebra $\triplet ^{A_m}$ is a $\mathbb{Z}_m$-orbifold of $\triplet$ and is generated by
$$ \omega, F^{(m)} = e ^{-m\alpha}, \ E^{(m)} = Q^{2m } F^{(m)}, H= Qe^{-\alpha};$$
for details see \cite{ALM-CCM}.
Therefore  $\triplet ^{A_m}$ is an extension of the singlet algebra and we have:
$$ \triplet ^{A_m} = \bigoplus _{n \in \Z}  \pi_{ 2 n m}. $$

\begin{theorem} \label{class-alm1} \cite{ALM-CCM} For every $0 \le i \le 2pm^2  -1$, there exists a unique (up to equivalence) irreducible $\triplet ^{A_m}$--module $L_{\frac{i}{m}}$ which decomposes as  the following  direct sum of  $\singlet$--modules
$$ L _{ \tfrac{i}{m} }  = \bigoplus_{ n \in {\Z} } M_{ \tfrac{i}{m}  - 2 p m n} . $$
Moreover, $ L_{ \tfrac{i}{m} } $ is realized as an irreducible subquotient of $V_{ \tfrac{i}{ 2 p m} \alpha + {\Z}(m \alpha) }$.

\end{theorem}

 \begin{definition}\label{sc-def} 
 We say that an irreducible  ${\N}$--graded $\overline{M(1)}$--module $W_1$ is a simple-current in the category $\N$--graded $\overline{M(1)}$--modules if for every irreducible $\N$-- graded module  $W_2$, there is a unique irreducible ${\N}$--graded $\overline{M(1)}$ --module $W_3$  such that the vector space of  intertwining operators   $I { W_3 \choose W_1 \ W_2}$ is $1$--dimensional and 
 $I { N \choose W_1 \ W_2} = { 0} $ for any other irreducible   ${\N}$--graded $\overline{M(1)}$--module  which is not isomorphic to $W_3$.  We write 
 $$W_1 \times W_2=W_3.$$
 \end{definition}

Note that we require  simple current property to hold only  on irreducible $\overline{M(1)}$--modules; studying fusion products (= tensor products) with indecomposable modules is a much harder problem which we do not address here.
 
\begin{conjecture} \label{conjecture-SC}  Let $p \in {\Z_{\ge 2}}$.
 Modules $\pi_j$, $j \in {\Z}$ are simple currents in the category of ordinary, ${\N}$--graded $\overline{M(1)}$--modules, and the following fusion rules holds:
 $$ \pi_j \times M_t  = M_{t -  j p}. $$
 \end{conjecture}

This conjecture is in agreement with the fusion rules result from \cite{CM} based on the Verlinde formula of characters. In Section \ref{fusion-rules-p2}, we shall prove this conjecture in the case $p=2$.
 

\begin{lemma}  \label{lemma-class} Assume that  Conjecture \ref{conjecture-SC}  holds.
Assume that $M$ is an irreducible $\triplet ^{A_m} $--module generated by a highest weight  vector $v_t$  for $\overline{M(1)}$ of the $(L(0), H(0))$-- weight
$\lambda_t$.
Then $t  \in  \frac{1}{m} {\Z}$.
\end{lemma}
\begin{proof}
By Conjecture  \ref{conjecture-SC}, we see that $\pi_{2m}$ sends $M_t := \overline{M(1)}v_t$ to $M_{t- 2mp }$, where $M_{t-2mp}=\overline{M(1)} v_{t- 2mp }$ is an irreducible $\overline{M(1)}$--module of  highest weight
$$ \left( \frac{1}{4p} {(t- 2mp  ) (t-  2mp -2p+2)}, {t - 2mp\choose 2p-1} \right). $$
But the difference between conformal weights of $v_t$ and $v_{t- 2mp }$ must be an integer. This leads to the following condition
 $$ m (t-p+1+m p) = \frac{1}{4p} {(t+2mp ) (t+ 2mp -2p+2)} -  \frac{1}{4p} {t (t-2p+2)} \in {\Z}. $$
 Thus, $mt \in {\Z}$.
\end{proof}

\begin{theorem}
Assume that  Conjecture \ref{conjecture-SC}  holds. Then the vertex operator algebra $\triplet ^{A_m} $ has precisely $2m^2 p$--irreducible modules, all constructed in \cite{ALM-CCM}.
\end{theorem}

\begin{proof}
Clearly, any irreducibe $\triplet ^{A_m} $-module $M$ is non--logarithmic (i.e., $L(0)$--acts semisimple). Such module has weights bounded from 
below, so $M$ must contain a singular vector $v_t$  for $\overline{M(1)}$. Then by using Lemma \ref{lemma-class} we get that as a $\overline{M(1)}$--module,
$$ M \cong \bigoplus_{ n \in \Z} M_{t  - 2m p} $$
where $t  \in  \frac{1}{m} {\Z}$ and  $ M_{t - 2m p}$ is the irreducible highest weight $\overline{M(1)}$--module with highest weight
$$ \lambda_{t-2mp}= \left( \frac{1}{4p} {(t- 2mp  ) (t-  2mp -2p+2)}, {t - 2mp\choose 2p-1} \right). $$
Now, the assertion follows from Theorem \ref{class-alm1} .
\end{proof}

\begin{remark}
A different approach to the problem of classification of irreducible $\triplet ^{A_m} $--modules was presented in \cite{ALM-SIGMA}. It is  based on certain constant term identities and it does not require any knowledge of fusion rules among $\singlet$-modules.
\end{remark}

  \section{ Fusion rules for $\singlet$ and a proof of Conjecture  \ref{conjecture-SC} for $p=2$. }
\label{fusion-rules-p2}

Let $V$ be a vertex operator algebra, $W_1, W_2, W_3$ be $V$--modules, and $\mathcal Y$ be an intertwining operator of type $$ {W_3 \choose W_1 \ \ W_2}.$$ 
For $u \in W_1$, we use $\mathcal Y(u,z) = \sum_{r \in {\Bbb C}} u_r z^{-r-1}$, $u_r \in {\rm Hom}(W_2,W_3)$.

Let $U_i$ be a  $V$--submodule of $W_i$, $i=1,2$. Define the inner {\em fusion product}
\bea  U_1 \cdot U_2 =\mbox{span}_{\Bbb C}  \{ u_r v \ \vert u\in U_1, v\in U_2, \ r \in {\Bbb C} \}. \label{fusion-product} \eea
Then $U_1 \cdot U_2$ is a $V$--submodule of $W_3$.

Let $M(1, \lambda)$ be as before. Then $M(1, \lambda)$ is an $\overline{M(1)}$--module and $e ^{\lambda}$ of highest weight
 $$ \left( \frac{1}{4p} {t (t-2p+2)}, {t \choose 2p-1} \right), \quad t= \langle \lambda, \alpha \rangle. $$
We specialize now $p=2$.  For $t \notin {\Z}$, the irreducible $\singlet$--module $M_t$ remains irreducible as a Virasoro module \cite{IK} and it can be realized as $M(1,\lambda)$.
 
 Denote by $A(\singlet)$ the Zhu associative algebra of $\singlet$. For a $\singlet$-module $M$, denote by $A(M)$ the $A(\singlet)$-bimodule with the left and right action given by
 \begin{align*}
 a * m  &={\rm Res}_x \frac{(1+x)^{{\rm deg}(a)}}{x} Y(a,x)m, \\ 
 m *a &={\rm Res}_x  \frac{(1+x)^{{\rm deg(a)}-1}}{x} Y(a,x)m,
 \end{align*}
 respectively.  It is known that there is an injective map from the space of intertwining operators of type ${ M_3 \choose M_1 \ M_2 }$ to 
  $${\rm Hom}_{A(V)} (A(M_1) \otimes_{A(V)} M_2(0),M_3(0))$$
 where $M_i(0)$ is the top degree of $M_i$. 
 Next, we analyze the space 
$$A(M_1) \otimes_{A(V)} M_2(0)$$
 in the case when $V=\singlet$, $M_1=M(1,\lambda)$ and $M_2=M(1,\mu)$. We denote by $L^{Vir}(c,0)$ the simple Virasoro vertex algebra 
 of central charge $c$. Clearly, $M(1,\lambda)$ can be viewed as a Virasoro vertex algebra module. As such, we have $A_{Vir}(M(1,\lambda)) \cong \mathbb{C}[x,y]$ as an $A(L^{Vir}(c,0))$-bimodule. Therefore, $A(M(1,\lambda)) \cong \mathbb{C}[x,y]/I$, where $I$ is the ideal coming from extra relations in $\singlet$. 

Observe also that inside $A(M(1,\lambda)) \otimes_{A(\singlet)} \mathbb{C}_{\mu}$, where $\mathbb{C}_{\mu}=\mathbb{C}v_\mu$ is one-dimensional $A(\singlet)$-module spanned by the highest weight vector $v_{\mu}$, 
we have the following relations (here $v \in A(M(1,\lambda))$):
\begin{align*}
v * [H] \otimes v_{\lambda} - { \langle \mu,\alpha \rangle \choose 3} v \otimes v_{\mu}& =0, \\
v *[\omega] \otimes v_{\lambda} - \frac{\langle \mu, \alpha \rangle (\langle \mu,\alpha \rangle -2)}{8} v \otimes v_{\mu} & =0, \\
{\rm Res}_x \frac{(1+x)^3}{x^2} Y(H,x) v=(H_{-2}+3H_{-1}+3H_{0}+H_1)v &  \in O(M(1,\lambda)).
\end{align*}

Let 
$$W(\lambda,\mu):=A(M(1,\lambda)) \otimes_{A(\singlet)} \mathbb{C}_\mu \cong   \mathbb{C}[x]/J.$$ Then we have
 
\begin{theorem}  \label{tipicni-kriterij} Suppose $t=\langle \alpha, \lambda \rangle \notin \mathbb{Z}$, and $s=\langle \alpha, \mu \rangle \notin \mathbb{Z}$. 
Then
\begin{align*}
p(x)=(x-\frac{(s+t)(s+t-2)}{8})(x-\frac{(s+t-2)(s+t-4)}{8}) \in J.
\end{align*}
In particular, $W(\lambda,\mu)$ is at most $2$-dimensional.
 \end{theorem}
\begin{proof}
Irreducibility of $M(1,\lambda)$ with respect to Virasoro algebra, implies that $H_{i} v_{\lambda}$, $i \ge -2$ can be written in the Virasoro basis. We omit these explicit  formulas here obtained by  a straightforward tedious computations. Together with 
$$(H_{-2}+3H_{-1}+3H_{-0}+H_1)v_{\lambda}=0,$$
we obtain a new relation $f(x,y)=0$ inside $A(M(1,\lambda))$. This relation turns into
{\tiny
$$\left(x-\frac{(t+s)(t+s-2)}{8}\right)\left( x-\frac{(t+s-2)(t+s-4)}{8}\right) \left( x-\frac{(s-t)(s-t-2)}{8}\right) \left( x-\frac{(s-t+2)(s-t)}{8}\right)=0$$
}
inside $W(\lambda,\mu)$.
Similarly, 
\bea 0 &= & v_{\lambda}  * [H] \otimes v_{\mu} - {s \choose 3} v \otimes v_{\mu} \nonumber \\ &=&  ( H_{-1} + 2 H_0 + H_1 -  {s \choose 3} ) v_{\lambda}  \otimes v_{\mu}   \nonumber \eea
turns into
{\small
$$\left(x-\frac{(t+s)(t+s-2)}{8}\right)\left( x-\frac{(t+s-2)(t+s-4)}{8}\right)\left(x-(\frac{s^2}{8}+\frac{t^2}{8}-\frac{st}{2}+\frac{s}{4}+\frac{t}{4})\right)=0$$
}
inside $W(\lambda,\mu)=0$.
Now we have to analyze common roots of these two polynomials. Clearly, they have 
$(x-\frac{(t+s)(t+s-2)}{8})( x-\frac{(t+s-2)(t+s-4)}{8})$ in common. It remains to analyse whether 
$$\frac{(s-t)(s-t-2)}{8}=\frac{s^2}{8}+\frac{t^2}{8}-\frac{st}{2}+\frac{s}{4}+\frac{t}{4}$$
or 
$$\frac{(s-t)(s-t+2)}{8}=\frac{s^2}{8}+\frac{t^2}{8}-\frac{st}{2}+\frac{s}{4}+\frac{t}{4}$$
The first equation holds if $t=2$ (for all $s$)  and also for $s=0$ (for all $t$). The second equation 
holds for $s=2$ (for all $t$) and also for $t=0$ (for all $s$). But this means either $s$ or $t$  are integral. The initial assumptions on $s$ and $t$ now yield the claim.
 \end{proof}

\begin{remark} By using exactly the same method we can prove that for $p=3$ the 
space $W(\lambda,\mu)$ is at most $3$-dimensional, as predicted in \cite{CM}.
\end{remark}

We shall use  the following result on intertwining operators for the Virasoro algebra.

\begin{lemma} \label{vir-fusion}
Assume that there is a non-trivial intertwining operator  $I(\cdot, z) $ of type
$$  { M(1, \nu) \choose L ^{ Vir} (-2, (n^2 + n  ) / 2)  \ \ M(1, \mu) } $$
such that $ I( v^{(n)} , z) = \sum_{r \in {\Bbb C} } v^{(n)}_r z ^{-r-1}$ and there is $r_0 \in {\Bbb C}$ such that
\bea &&  v^{(n)} _{r_0} v_{\mu} = \lambda v_{\nu}, \quad  (\lambda \ne 0). \label{uv-int} \eea
(Here $v^{(n)}$ is highest weight vector in $L ^{Vir} (-2, (n^2 + n  ) / 2)$).
Then $$ \nu  = \mu  + (i-n /2  ) \alpha,  \quad \mbox{for}  \quad   0 \le i \le n. $$
\end{lemma}
\begin{proof} This is essentially proven in Theorem 5.1 in \cite{Lin}.
\end{proof}

  \begin{theorem} \label{SC} Let $p=2$.
  Modules $\pi_j$, $j \in {\Z}$ are simple currents in the category of ordinary, ${\N}$--graded $\overline{M(1)}$--modules.
  \end{theorem}
  \begin{proof}  
  Let us here consider the case $j >  0$ (the case $j<0$ can be treated similarly). Then $t = -2 j \notin \{0, 2\}$. Let as above $M_s$ and $M_r$ be irreducible $\singlet$--modules with highest weight $\lambda_s$ and $\lambda_r$, respectively.
   By using the proof of Theorem  \ref{tipicni-kriterij}  we see that if
   $$ I { M_r  \choose \pi_j  \ \  M_s } \ne 0, $$
   then exactly one of the following holds:
    \begin{itemize}
    \item[(1)] $s \notin \{ 0, 2 \}$ and $r \in \{ s- 2 j, s-2 j -2 \}$,
    \item[(2)] $s = 0$, $ r \in \{   2j,  2j +2, 2j+4, -2j-2, - 2j, -2j +2    \} $,  
    \item[(3)] $ s = 2$, $ r \in   \{ 2j,  2j +2, 2j+4,  -2-2j, - 2j ,  -2j +2   \}  . $
    \end{itemize}
    Cases (2) and (3) come from analyzing additional solutions. By using the following  relation in $A (M(1, \lambda))$
    $$ H * v_{\lambda} - v_{\lambda} * H = ( H_0  + 2 H_1 + H_2 ) v_{\lambda}  $$
    we get  that 
      $$ r = s + t \quad \mbox{or} \quad  r = s + t -2  \quad \mbox{or} \quad
  r = 1/3 (5 - 3 s + t \pm  \sqrt{1 + 12 s + 4 t - 12 s t + 4 t^2}). $$
This implies 
     \begin{itemize}
    \item[(1')] $s \notin \{ 0, 2 \}$ and $r \in \{ s- 2 j, s-2 j -2 \}$,
    \item[(2')] $s = 0$,  $r \in \{-2j -2, -2j,  -2j +2 , \frac{ 4+ 2j}{3} \}$, 
    \item[(3')] $ s = 2$,  $r  \in \{ -2j  -2,  -2j , -2j+ 2, \frac{ 4 + 2j}{3} \}$.
    \end{itemize}
    

  On the other hand since $L ^{ Vir} (-2, (j^2 + j  ) / 2)$ is an irreducible Virasoro submodule of $\pi_j$, then in the category of $L^{Vir} (-2, 0)$--modules we have a non-trivial  intertwining operator of the type
  $$ {M_r \choose L ^{ Vir} (-2, (j^2 + j  ) / 2)  \ \ M_s }. $$
  Now, applying Virasoro fusion rules computed in Lemma \ref{vir-fusion}  we get $$r \in \{  s -2 j, s-2j +4, \dots, s+2j -4, s+ 2j\} . $$ 
  By analyzing all cases above, we see that the only possibility is $ r = s-2j $.
  The proof follows.
 
  \end{proof}

\section{Deformed realization of the triplet and singlet vertex algebra}

\subsection{Definition of deformed realization}
In this section we present a new realization of the triple and doublet vertex operator algebra. This is then used to show that the sufficient conditions 
obtained in Theorem \ref{tipicni-kriterij} are in fact necessary.
Let $p \in {\Z}_{>0}$.
Let $V_L$ be the lattice vertex algebra associated to the lattice
$$L= \Z {\alpha_1} + \Z\alpha_2, \quad \langle \alpha_1, \alpha_1 \rangle = 2p -1, \quad \langle \alpha_2, \alpha_2 \rangle = 1,  \quad  \langle \alpha_1, \alpha_2 \rangle = 0.$$
Let $\alpha = \alpha_1 + \alpha_2$. Define
$$ \omega = \frac{1}{4p}  \alpha(-1) ^2 + \frac{p-1}{2p} \alpha (-2)   + \frac{p-1}{p} e ^{-2\alpha_2}.$$
Then $\omega $ generates the Virasoro vertex operator algebra  $L(c_{1,p}, 0)$ of central charge
$c_{1,p} = 1 - 6 (p-1) ^2 / p$. More generally,
\begin{proposition}
Let $\omega' \in V$ be a conformal vector of central charge $c$ and $v \in V_1$ (degree one) a primary vector for $\omega'$ such that $[v_n,v_m]=0$ for all $m,n \in \mathbb{Z}$. Then 
$$\omega=\omega'+v$$ is a conformal vector of the same central charge.
\end{proposition}

Consider the following operator
$$ S  = e^{\alpha_1+\alpha_2} _0 + e^{\alpha_1- \alpha_2 }_0 $$
acting on $V_L$.
\begin{lemma} \label{S-operator}
The operator $S$ is a screening operator for the Virasoro algebra generated by $\omega$.
\end{lemma}
\begin{proof}  We have to prove that 
$$[e^{\alpha_1+\alpha_2}_0+e^{\alpha_1-\alpha_2}_0,L'(n)+ \frac{p-1}{p} e^{-2 \alpha_2}_{n+1}]=0.$$
This follows directly from the relations
\begin{align*}
& [e^{\alpha_1+\alpha_2}_0,L'(n)]=0 , \quad L'(0)e^{\alpha_1-\alpha_2}=0 \\
& [e^{\alpha_1-\alpha_2}_0,e^{-2 \alpha_2}_{n+1}]=0,   \quad  [e^{\alpha_1+\alpha_2}_0, e^{-2 \alpha_2}_{n+1}]=\left( (\alpha_1(-1)+\alpha_2(-1) ) e^{\alpha_1-\alpha_2}\right)_{n+1} \\
& [L(n)',e^{\alpha_1-\alpha_2}_0]=(L(-1)' e^{\alpha_1-\alpha_2})_{n+1}=\frac{p-1}{p} \left( (\alpha_1(-1)+\alpha_2(-1)) )e^{\alpha_1-\alpha_2 } \right)_{n+1}.
\end{align*}
\end{proof}

Let $ \widetilde { \mathcal{W} }(p)$ be the  vertex subalgebra of $V_L  $  generated by
$$\omega, \ F= e^{-\alpha}, \ H= S F, \ E = S ^2 F. $$
Let $\widetilde {\mathcal A} (p)$ be the  (generalized) vertex algebra  generated by
\begin{align*}
a ^- &= e ^{-\alpha/2}, \\ 
a^+ & = S a^- = S_{p-1} (\alpha)  e^{\alpha /2} + S_{p-2} (\alpha_1 - \alpha_2)  e^{\alpha /2 - 2 \alpha_2}. 
\end{align*}
Clearly, $ \widetilde { \mathcal{W} }(p) \subset \widetilde {\mathcal A} (p).$  Let $\widetilde {\mathcal{W} }^0 (p)$ be the vertex subalgebra of $ \widetilde { \mathcal{W} }(p) $ generated  by 
$\omega$ and $$  H= S F = S_{2p-1}(\alpha) + S_{2p-3} (\alpha_1 -\alpha_2) e ^{-2\alpha_2}. $$

\begin{conjecture} \label{conjecture-deformed}
We have
$$ \widetilde {\mathcal A} (p) \cong {\mathcal A}  (p),  \quad \widetilde {\mathcal W} (p) \cong {\mathcal W}  (p),  \quad \widetilde {\mathcal W} ^ 0 (p) \cong  \singlet, $$
as vertex algebras.
\end{conjecture}

\begin{theorem} \label{conj-proof-23}
Conjecture \ref{conjecture-deformed} holds for $p=2$ and $p=3$.
\end{theorem}
\begin{proof}
Let 
\begin{align*}
a ^- &= e ^{-\alpha/2}, \\  a^+ & = S a^- = S_{p-1} (\alpha)  e^{\alpha /2} + S_{p-2} (\alpha_1 - \alpha_2)  e^{\alpha /2 - 2 \alpha_2}. 
\end{align*}
{\bf Case $p=2$.}

Then
$$ a^+ =  \alpha(-1) e ^{\alpha /2}  +  e^{\alpha /2 - 2 \alpha_2}. $$
Direct calculation  shows  that
$$ a^- _{n } a^+ =  2 \delta_{n,1} {\bf 1}, \quad   a^- _{n} a^+  = 0   \quad (n \ge 0), $$
and by (super)commutator formulae we have
\bea  \{ a^- _{n },  a^+ _{m}\}  = 2 n \delta_{n+m,0}, \ \{ a^{\pm}  _{n },  a^{\pm}  _{m}\} = 0 \qquad (n, m \in \Z).  \label{sf-com} \eea

This implies 
$a ^-$ and $a^+$ generate the subalgebra isomorphic to the symplectic fermion vertex superalgebra $\mathcal A(2)$ (cf. \cite{Abe}, \cite{K}, \cite{R}). Since even part of the symplectic fermions is isomorphic to the triplet vertex algebra $\mathcal W(p)$ for $p=2$, the assertion holds.

{\bf Case $p=3$.}
Let $F_{-p/2}$ be the generalized lattice vertex algebra associated to the lattice $L= \Z \varphi$, $\langle \varphi, \varphi \rangle = -p/2$ with generators $e ^{\pm \varphi}$  (cf. \cite{DL}) and $V^{-4/3}(sl_2)$ the universal affine vertex algebra at the admissible level $-\frac{4}{3}$.
As in  \cite{A-2005}, one shows that  for $p=3$ 
\bea &&  e= a^- \otimes e^{-\varphi} ,  \  f= -\frac{2}{9} a^+ \otimes e^{\varphi}, \  h =  \frac{4}{3} \varphi  \label{efh} \eea
defines a non-trivial homomorphism  $$\Phi: V^{-4/3}(sl_2) \rightarrow \widetilde{\mathcal A} (3) \otimes F_{-3/2}$$ 
which maps singular vector in $V^{-4/3}(sl_2) $ to zero (calculation  is completely analogous as in \cite{A-2005}). Therefore (\ref{efh}) gives an embedding of the simple affine vertex algebra $V_{-4/3} (sl_2)$ into $ \widetilde{\mathcal A} (3) \otimes F_{-3/2}$.    By construction,  we have that the vacuum space
  $$\Omega (V_{-4/3} (sl_2))=\{ v  \in V_{-4/3} (sl_2)  \ \vert  h(n) v = 0 \ \forall n \ge 1 \} $$  is generated by $a^{+}$ and $a^{-}$ and therefore  $$\Omega (V_{-4/3} (sl_2))\cong  \widetilde {\mathcal A} (3). $$
 In  \cite{A-2005},  it was proven that
 \begin{itemize}
\item The  vacuum space 
$\Omega (V_{-4/3} (sl_2))$ is  isomorphic to the doublet vertex algebra $\mathcal A(3)$.
\item The parafermionic vertex algebra $$ K(sl_2 , -4/3) = \{ v \in V_{-4/3} (sl_2) \ \vert \  h(n) v = 0 \ \forall n \ge  0\} $$
is isomorphic to    ${\mathcal W} ^ 0 (3) \cong  \singlet. $
\end{itemize}
So we have
\bea  \widetilde {\mathcal A} (3) &  \cong &  \Omega (V_{-4/3} (sl_2))  \cong  \mathcal A(3),  \nonumber \\
 \widetilde { \mathcal W} ^ 0 (3)  & \cong  &  K(sl_2 , -4/3)  \cong  {\mathcal W} ^ 0 (3). \nonumber  \eea
Since the triplet $\mathcal W(3)$ is a $\Z_2$-orbifold of ${\mathcal A} (3) $, we also have that
$  \widetilde { \mathcal W}  (3)     \cong  {\mathcal W}  (3). $
This proves the conjecture  in the case $p=3$.
\end{proof}

\subsection{ On the proof of Conjecture \ref{conjecture-deformed} for  $p\ge 4$} 
At this point, we do not have yet a uniform proof of Conjecture \ref{conjecture-deformed}  for all $p$.  The problem is that $\triplet$ and $\singlet$ are not  members of a family of generic  vertex algebras and $\mathcal W$--algebras. In order to solve this difficulty,  one can     study the embeddings of  $\triplet$ and $\singlet$  into  affine vertex algebras and minimal $\mathcal W$--algebras (cf. \cite{A-2005}, \cite{CRW}).   The case $p=3$ was alredy discussed in   the proof of Theorem \ref{conj-proof-23}.
A similar arguments can be repeated  in the $p=4$ case. 
For $p=4$ one  can show that 
\begin{equation} \label{gen}
  \qquad a^+ \otimes e^{\varphi}, \ \  a^- \otimes e^{-\varphi} 
  \end{equation}
generate the simple  vertex algebra $\mathcal W_3 ^{(2)}$ with central charge $-23/2$ whose vacuum space is generated by $a^{+}$ and $a^{-}$. 
 By  using the realization from \cite{CRW}, we get that $a^{+}$ and $a^{-}$ generate the doublet vertex algebra $\mathcal A(4)$. So  Conjecture  \ref{conjecture-deformed} also holds for $p=4$.
For $p \ge 5$, one needs to prove that (\ref{gen}) define the Feigin-Semikhatov $W$--algebra $W_p ^{(2)}$ embedded  into $\mathcal A(p) \otimes F_{-p/2}$.  But this calculation is much more complicated.


 \subsection{ Deformed action for $p=2$}

Let now $p=2$.  Then the singlet vertex algebra $\singlet$ is isomorphic to a subalgebra of $V_L$ generated by
\begin{align*}
\omega & = \frac{1}{8}  \alpha(-1) ^2 + \frac{1}{4} \alpha (-2)   + \frac{1}{2} e ^{-2\alpha_2}, \\ 
H & = S e^{-\alpha} = S_3 (\alpha)  + (\alpha_1 (-1) - \alpha_2 (-2) ) e^{-2\alpha_2}. 
\end{align*}

 Let $r, s  \in {\Bbb C}$. Consider  the  $ \singlet $--module $$\mathcal F (r, s) := \singlet. v_{r,s} , \quad v_{r,s}=e ^{\tfrac{ r}{3 } \alpha_1 + s  \alpha_2}.$$
We have:
\bea L(0) v_{r.0} &=&   h_{r+1, 1} v_{r,0}, \nonumber \\    L(0) v_{r,-1} &=& h_{r ,1} v_{r,-1}  \nonumber  \\ L(0) v_{r,1} &=& h_{r+ 2 ,1} v_{r,1}  + \frac{1}{2} v_{r, -1} , \nonumber  \\
L(0) (v_{r,1} -  \frac{1}{2  h_{r,1} } v_{r,-1} ) & = & h_{r+2,1} v_{r,1},   \nonumber \\
 H(0) v_{r.0} &=& {r \choose 3 } v_{r,0} ,\nonumber \\
 H(0) v_{r,-1} & = & { r-  1  \choose 3 } v_{r,-1} ,\nonumber \\
 H(0) v_{r,1} & = & { r+  1  \choose 3} v_{r, 1} + (r-1)  v_{r,-1} ,   \nonumber 
\eea 
where
$$ h_{r ,1} = \frac{( r-1)^2 }{8} - \frac{(r-1) }{4} =\frac{ (r-2) ^2 - 1}{ 8} . $$

\begin{proposition} \label{log-mod}
\item[(1)] Assume that $r \notin \Z $. Then
$$ \mathcal P_r:=\singlet  . v_{r,1} = \singlet . v^ + _{r,1} \bigoplus \singlet . v^ -  _{r,-1} $$
where
$$ v ^ + _{r,1} = v_{r,1} + \frac{p-1}{p} \frac{1}{h_{r+2,1} - h_{r,1} } v_{r,-1} = v_{r,1} + \frac{1}{r-1} v_{r,-1}, \  v^-   _{r,1}  = v_{r,-1}$$
are highest weight vectors  for $\singlet  $ of $(L(0), H(0))$ weight $$ \left(h_{r+2,1},  { r+  1  \choose 3 }\right), \quad  \left(h_{r,1},  { r-  1  \choose 3 }\right) $$
respectively. In particular,
$$  \mathcal P_r = M_{r+1} \bigoplus M_{r-1}. $$
\item[(2)] Assume that $r=p-1 =1$. Then
$ \mathcal P_{1} :=\singlet . v_{1,1}$ is a $\Z_{\ge 0}$-graded logarithmic  $\singlet$--module. The lowest component $\mathcal P_{1}(0)$ is $2$--dimensional,  spanned by $v_{1,1}$ and $v_{1,-1}$. For $r \in \mathbb{Z} \setminus \{0\}$, see also Remark \ref{log-int}.
\end{proposition}
\begin{proof}
Parts (1) and (2) follow immediately due to irreducibility of $M_t$  and because $L(0)$ forms a Jordan block on the top component, respectively. 
\end{proof}

\begin{remark}
In \cite{AdM-2009}, the authors presented an explicit realization of certain logarithmic modules for $\triplet$. Our  deformed realization of $\triplet$ gives an alternative realization of these logarithmic modules. We plan to return to these realizations in our forthcoming publications.
\end{remark}

\section{Intertwining operators between typical $\singlet$ and $\triplet ^{A_m}$--modules: the $p=2$ case}

In this section we fix $p=2$.
\begin{theorem}
\item[(1)] Assume that $t_1, t_2, t_1 + t_2 \notin {\Z}$. Then in the category of $\singlet$--modules, there exists non--trivial  intertwining operators of types
$${ M_{t_1+t_2} \choose M_{t_1} \ \ M_{t_2} }, \quad { M_{t_1+t_2-2} \choose M_{t_1} \ \ M_{t_2} }. $$
In particular, the following fusion rules holds:
$$ M_{t_1} \times M_{t_2} = M_{t_1+t_2} + M_{t_1 + t_2 -2}. $$
\end{theorem}
\begin{proof} In our realization, $\singlet$ is a vertex subalgebra of the vertex algebra $U:=M_{\alpha_1 } (1) \otimes V_{\Z  2 \alpha_2}$, where $M_{\alpha_1 } (1)$ is the Heisenberg vertex algebra, generated by $\alpha_1$, and $V_{\Z  2 \alpha_2}$ is the lattice vertex algebra associated to the even lattice $\Z (2 \alpha_1)$.
First we notice that in the category of $U $--modules we have a non-trivial intertwining operator $\mathcal Y$ of type
$$ { M_{\alpha_1} (1,  \frac{t_1 + t_2-1}{3}  \alpha_1)  \otimes V_{\alpha_2 + \Z  2 \alpha_2 }   \choose  M_{\alpha_1} (1,  \frac{t_1 }{3}\alpha_1 )  \otimes V_{ \Z  (2 \alpha_2 ) }  \ \ M_{\alpha_1} 
(1,  \frac{t_2-1 }{3} \alpha_1 ) \otimes V_{\alpha_2 + \Z  (2 \alpha_2)  }    }. $$
Clearly,  \bea  M_{t_1} &=& \singlet. v_{t_1,0}  \subset M_{\alpha_1} (1,  \frac{t_1 }{3}\alpha_1 )  \otimes V_ { \Z  (2 \alpha_2 )}  \nonumber \\
M_{t_2} &=& \singlet. v ^+ _{t_2-1, 1} \subset M_{\alpha_1} (1,  \frac{t_2-1 }{3}\alpha_1 )  \otimes V_{\alpha_2 + \Z  2 \alpha_2 }  \nonumber \\
M_{t_1 + t_2} &=& \singlet. v ^+ _{t_1+ t_2-1, 1} \subset M_{\alpha_1} (1,  \frac{t_1+ t_2-1 }{3}\alpha_1 )  \otimes V_{\alpha_2 + \Z  2 \alpha_2 }  \nonumber \\ 
M_{t_1 + t_2-2} &=& \singlet. v ^- _{t_1+ t_2-1, 1} \subset M_{\alpha_1} (1,  \frac{t_1+ t_2-1 }{3}\alpha_1 )  \otimes V_{\alpha_2 + \Z  2 \alpha_2 }  \nonumber 
\eea
By abusing the notation, let $\mathcal{Y}$ denotes the restriction to $M_{t_1}$ as defined above. We have to first 
check that $\mathcal{Y}$ satisfies the $L(-1)$-property for the new Virasoro generator $L(-1)=L'(-1)+\frac{1}{2}e^{-2 \alpha_2}_{0}$. 
Observe that $\mathcal{Y}(L(-1)v_{t_1,0})=\mathcal{Y}(L'(-1)v_{t_1,0},x)=\frac{d}{dx} \mathcal{Y}(v_{t_1,0},x)$.
For other vectors, the $L(-1)$ property follows from the fact that $M_{t,1}$ is a cyclic module.

By using Theorem  \ref{tipicni-kriterij}   and the fact that $M_{t_1}.M_{t_2}$ is completely reducible, we see that the corresponding fusion product  (\ref{fusion-product}), defined via the intertwining operator $\mathcal Y$, must satisfy  $M_{t_1} \cdot M_{t_2} \subset M_{t_1+t_2} \oplus M_{t_1+t_2-2}$. 
By construction, we have that \bea w&=& v  _{t_1+ t_2-1, 1} + \frac{1}{t_2-2} v  _{t_1+ t_2-1, -1}  \nonumber \\
&=&   v ^{+} _{t_1+ t_2-1, 1} + \frac{t_1}{(t_2-2)(t_1 + t_2 -2)} v ^{-} _{t_1+ t_2-1, 1} \in M_{t_1} \cdot M_{t_2}, \nonumber \eea
which implies that  singular vectors $v ^{\pm} _{t_1+ t_2-1, 1}$ belong to $M_{t_1} \cdot M_{t_2} $. So  we get that $M_{t_1} \cdot M_{t_2} = M_{t_1+t_2} \oplus M_{t_1+t_2-2}$ as desired.
%
 \end{proof}

\begin{remark} \label{log-int}
One can show that for $t_1,t_2 \notin \mathbb{Z}$, $t_1 + t_2 \in {\Z}$, our free field realizations in Proposition \ref{log-mod} also gives the existence of  a logarithmic intertwining operator of type
$$ { \mathcal P_{t_1+ t_2} \choose M_{t_1} \ \ M_{t_1} }. $$
Moreover, observe that the condition $t_1,t_2 \notin \mathbb{Z}$ is not necessary here and we may assume 
$t_1, t_2 \in \{ 4n+1, 4n-1 : n \in \mathbb{Z} \}$.  This leads to a new family of intertwining operator presumably obtained by the restriction 
of an intertwining operator coming from a triple of symplectic fermions modules. We shall study  these intertwining operators in more details in our future
publications.
\end{remark}


We have similar result in the category of $\triplet ^{A_m}$--modules:
\begin{theorem}
Assume that $i_1, i_2,  i_3 \in \{0, \dots, 2pm^2  -1\} $ $i_1, i_2,  i_1 + i_2\notin  m{\Z}$, where $m$ is a positive integer. Then in the category of $\triplet ^{A_m}$--modules, there exists non--trivial  intertwining operators of type
$${ L_{i_3/m}  \choose L_{i_1/m} \ \ L_{i_2/m} }  $$
if and only if $$ i_3 \equiv i_1 + i_2 \ \   \mbox{mod} (2m^2 p)\qquad \mbox{or} \ \ i_3 \equiv i_1 + i_2  - 2m \ \   \mbox{mod}( 2m^2  p). $$
\end{theorem}

\begin{remark} Results in this paper, together with Remark \ref{log-int},  provides a rigorous proof of the main conjecture in \cite{CM} pertaining to fusion rules of 
$\singlet$ for $p=2$. 

Almost all results in the paper can be extended for the $N=1$ singlet/triplet vertex algebras introduced and studied by the authors \cite{AdM-CMP}
(work in progress).
\end{remark}

\end{document}